\documentclass[11pt]{amsart}
\usepackage{geometry}                
\usepackage{graphicx}
\usepackage{amssymb}
\usepackage{extarrows}
\usepackage{epstopdf}
\usepackage{color}
\usepackage{bbm, dsfont}
\newtheorem{theorem}{Theorem}
\newtheorem{lemma}{Lemma}[section]
\newtheorem{corollary}{Corollary}[section]
\newtheorem{proposition}{Proposition}[section]

\theoremstyle{definition}
\newtheorem{definition}{Definition}
\theoremstyle{remark}
\newtheorem{remark}{Remark}

\theoremstyle{example}

\usepackage{latexsym,amssymb,amsmath,amsfonts,amsthm}
\usepackage{amsthm,amsmath}
\usepackage{float}
\usepackage{enumerate}
\usepackage{epic}
\usepackage{fullpage}
\usepackage{bbm}
\usepackage{subfigure}
\usepackage{graphicx}
\usepackage[toc,page]{appendix}
\usepackage{amsfonts}
\usepackage[all]{xy}
\usepackage{caption}
 \usepackage{geometry}              
\usepackage{graphicx}
\usepackage{enumerate}
\usepackage{amssymb}
\usepackage{epstopdf}
\usepackage{color}
\usepackage{bbm, dsfont}
\usepackage{hyperref}
\usepackage[utf8]{inputenc}
\usepackage{amssymb,amsthm,amsmath,amssymb,wrapfig,dsfont}
\usepackage[dvipsnames]{xcolor}
\definecolor{myred}{RGB}{251,154,133}
\definecolor{myblue}{RGB}{153,206,227}
\definecolor{mylightblue}{RGB}{0, 150, 255}
\definecolor{mygreen}{RGB}{32, 210, 64}
\definecolor{mygray}{RGB}{220, 220, 220}

\usepackage{tikz}
\usetikzlibrary{decorations.pathmorphing}
\tikzset{snake it/.style={decorate, decoration=snake}}
\usetikzlibrary{shapes.geometric,positioning,decorations.pathreplacing}

\DeclareFontFamily{OML}{rsfs}{\skewchar\font'177}
\DeclareFontShape{OML}{rsfs}{m}{n}{ <5> <6> rsfs5 <7> <8> <9>
rsfs7 <10> <10.95> <12> <14.4> <17.28> <20.74> <24.88> rsfs10 }{}
\DeclareMathAlphabet{\mathfs}{OML}{rsfs}{m}{n}

\newcommand{\BR}{{\mathbb{R}}}

\newcommand{\BZ}{{\mathbb{Z}}}

\newcommand{\CC}{{\mathcal{C}}}

\newcommand{\CF}{{\mathcal{F}}}

\newcommand{\CI}{{\mathcal{I}}}

\newcommand{\ind}{{\mathbbm{1}}}

\newcommand{\bae}{\begin{equation}\begin{aligned}}
\newcommand{\eae}{\end{aligned}\end{equation}}

\newcommand{\ep}{{\epsilon}}

\DeclareFontFamily{OML}{rsfs}{\skewchar\font'177}
\DeclareFontShape{OML}{rsfs}{m}{n}{ <5> <6> rsfs5 <7> <8> <9>
	rsfs7 <10> <10.95> <12> <14.4> <17.28> <20.74> <24.88> rsfs10 }{}
\DeclareMathAlphabet{\mathfs}{OML}{rsfs}{m}{n}

\newcommand{\fw}{\text{W}^{[0, \infty)}}

\newcommand{\zd}{\mathbb{Z}^d}

\newcommand{\vt}{v^{(T)}}

\usepackage{amsmath}
\usepackage{mathrsfs}

\begin{document}
\title{Percolation for the Finitary Random interlacements}

\author{Eviatar B. Procaccia}
\address[Eviatar B. Procaccia]{Texas A\&M University \& Technion - Israel Institute of Technology}
\urladdr{www.math.tamu.edu/~procaccia}
\email{eviatarp@gmail.com}

\author{Jiayan Ye}
\address[Jiayan Ye]{Texas A\&M University}
\urladdr{www.math.tamu.edu/~tomye}
\email{tomye1992@gmail.com}

\author{Yuan Zhang}
\address[Yuan Zhang]{Peking University}
\email{zhangyuan@math.pku.edu.cn}

\maketitle

\begin{abstract}
	In this paper, we prove a phase transition in the connectivity of finitary random interlacements $\mathcal{FI}^{u,T}$ in $\BZ^d$, with respect to the average stopping time $T$. For each $u>0$, with probability one $\mathcal{FI}^{u,T}$ has no infinite connected component for all sufficiently small $T>0$, and a unique infinite connected component for all sufficiently large $T<\infty$. This answers a question of Bowen \cite{bowen2017finitary} in the special case of $\BZ^d$.
\end{abstract}

\section{Introduction}
\label{section_intro}
The model of random interlacements (RI) was introduced by Sznitman in \cite{Sznitman2009Vacant}, and finitary random interlacements (FRI) was recently introduced by Bowen \cite{bowen2017finitary} to solve the Gaboriau-Lyons problem in the case of arbitrary Bernoulli shifts over a non-amenable group. The Gaboriau-Lyons problem \cite{gaboriau2009measurable} asks whether every non-amenable measured equivalence relation contains a non-amenable treeable subequivalence relation. Bowen \cite{bowen2017finitary} gave a positive answer for the special case by studying FRI. Informally speaking, FRI can be described as a cloud of geometrically killed random walks on $\zd$. Similar to the convention that the range of random interlacements (RI) at level $u > 0$ is denoted by $\mathcal{I}^u$, the range of FRI is denoted by $\mathcal{FI}^{u,T}$, where $u >0$ is the multiplicative parameter controlling the number of geometrically killed random walks, and the parameter $T >0$ is the expected length of a geometrically killed random walk.

In this paper, we are interested in the FRI in the lattice $\zd$, with $d \geq 3$.  In \cite{bowen2017finitary} Bowen showed that FRI measure converges to RI measure in the weak* topology as $T$ goes to infinity. Thus it is natural to compare the geometry, especially the connectivity properties of the two systems. For any two vertices $x,y \in \mathcal{FI}^{u,T}$, $x$ and $y$ are said to be connected if there exist vertices $x_0, x_1, \cdots, x_n \in \mathcal{FI}^{u,T}$ such that $x = x_0$, $y =x_n$, and $(x_i, x_{i+1})$ are edges in the graph $\mathcal{FI}^{u,T}$ for all $0 \leq i < n$.  

In \cite{Sznitman2009Vacant} Sznitman proved that $\mathcal{I}^u$ is almost surely connected. In \cite{Procaccia2011Geometry} and \cite{R2010Connectivity}, it is shown that for any two vertices $x, y \in \mathcal{I}^u$, there is a path between $x$ and $y$ via at most $\lceil d/2 \rceil$ random walk trajectories, and this bound is sharp. This does not hold for FRI since for each site $x\in \BZ^d$ there is always a positive probability that $x$ is an isolated point in $\mathcal{FI}^{u,T}$. 

In \cite{bowen2017finitary}, Bowen proved the existence of infinite connected components within $\mathcal{FI}^{u,T}$ for large $T$ in all non-amenable groups. He raised the question that, whether there are infinite connected component(s) within $\mathcal{FI}^{u,T}$ for each $u>0$ and sufficiently large $T$ in any amenable group. See Question 2, \cite{bowen2017finitary} for details. In this paper, we give a partial affirmative answer to this question by showing there exists a phase transition for the FRI in $\BZ^d$. For any $u >0$, there are $0 < T_0(u,d) \leq T_1 (u,d) < \infty$. If $T > T_1$, $\mathcal{FI}^{u,T}$ has a unique infinite cluster almost surely. If $0 < T < T_0$, $\mathcal{FI}^{u,T}$ has no infinite cluster almost surely. To be precise, we have 

\begin{theorem}[Supercritical Phase]
	\label{main1}
	For all $u >0$, there is a $0 < T_1 (u,d) < \infty$ such that for all $T > T_1$, $\mathcal{FI}^{u,T}$ has an unique infinite cluster almost surely.
\end{theorem}

\begin{theorem}[Subcritical Phase]
	\label{main2}
	For all $u >0$, there is a $0< T_0 (u,d) < \infty$ such that for all $0 < T < T_0$, $\mathcal{FI}^{u,T}$ has no infinite cluster almost surely. 
\end{theorem}

\begin{remark}
In this paper, we consider percolation of $\mathcal{FI}^{u,T}$ as percolation for the edges crossed by trajectories in the FRI process. The notion of connectivity is defined in the second paragraph of this section. In literature, one usually considers percolation of interlacements as percolation for the vertices (sites) hit by the random interlacements process. The proof of Theorem \ref{main2} relies on the kind of percolation we choose, whereas the proof of Theorem \ref{main1} holds for both bond and site percolation. 
\end{remark}

The proof of Theorem \ref{main1} relies on a renormalization/block construction argument along with coupling the FRI to RI. We define a good block event in Section \ref{sec good def}, and we prove that this good event occurs with high probability in Section \ref{sec good with high probability}. In Section \ref{sec renormal} we apply a standard renormalization/block construction argument to see the spread of our ``good blocks" dominates a supercritical percolation. The proof of uniqueness is presented in Section \ref{sec unique}.  The proof of Theorem \ref{main2} is presented in Section \ref{sec subcritical}.

After the paper was posted on arXiv, we learned about works \cite{ MR3687239, MR3519252} considering a relevant continuum percolation model. In this model, a Poisson cloud of Brownian motion paths ($d=2,3$), or Wiener sausages with radius $r$ ($d\ge 4$), both truncated at some finite time $t$, are sampled on $\BR^d$ according to intensity measure $\lambda Leb(\BZ^d)$, for some fixed $\lambda>0$. \cite{MR3687239, MR3519252} proved the existence of a percolation phase transition with respect to $t$, and the asymptotic behavior of the critical value in $d\ge 4$ as $r\to0$. 

The results we prove here for finitary interlacements may, at least superficially, well resemble some discrete version of their problem. However, as pointed out in Question (3) and (4), \cite{MR3687239}: the rigorous relations between their model and random interlacements or ``the system of independent finite-time random walks, which are initially homogeneously distributed on $\zd$" remain open problems. It was conjectured in \cite{MR3687239, MR3519252} that the continuum model will bear more similarities to a continuous version of random interlacements \cite{MR3167123} when $\lambda\to 0$, $t\to \infty$, while $\lambda t$ remains a constant. Heuristically, this also agrees with the setting in finitary interlacements, see Definition \ref{def_FRI} and \ref{informal} for details.

\subsection{Open problems}

The phase transition for FRI is not entirely understood. We expect that there is a critical value $0 < T_c (u,d) < \infty$ such that $\mathcal{FI}^{u,T}$ has an infinite cluster for all $T > T_c$ and no infinite cluster for all $T < T_c$. Equivalently, $T_1 (u,d) = T_0 (u,d)$ in Theorems \ref{main1} and \ref{main2}. We are unable to prove a sharp phase transition in this paper. It is unclear that whether $\mathcal{FI}^{u,T}$ is monotonic with respect to $T$. By Definition \ref{informal}, as $T$ increases, the average number of geometrically killed random walks that each vertex generated decreases, but the average length of each geometrically killed random walks increases. Therefore, unlike other percolation models, one cannot prove a sharp phase transition for FRI using monotonicity.

Given Theorem \ref{main1} it is natural to ask about the internal graph distance in the unique infinite cluster. In the case of random interlacements it was proved in \cite{vcerny2012internal,drewitz2014chemical,procaccia2014range} that the internal graph distance in RI is proportional to the $\BZ^d$ distance with high probability. It would be interesting to show a similar result for the internal graph distance in the unique infinite cluster of $\mathcal{FI}^{u,T}$, for large enough $T>0$. Moreover if we denote by $d_{\mathcal{FI}^{u,T}}(\cdot,\cdot)$ and $d_{\mathcal{I}^{u}}(\cdot,\cdot)$ the internal graph distances in FRI and RI, one can conjecture that for every $u>0$, 
$$
\lim_{T\rightarrow\infty}\lim_{\|x\|_1\rightarrow\infty}d_{\mathcal{FI}^{u,T}}([0],[x])/\|x\|_1=\lim_{\|x\|_1\rightarrow\infty}d_{\mathcal{I}^{u}}([0],[x])/\|x\|_1
,$$
where $[x]$ denotes the closest vertex in the appropriate infinite component to $x\in\BZ^d$. A relative question is the continuity of the function $$u\to\lim_{\|x\|_1\rightarrow\infty}d_{\mathcal{FI}^{u,T}}([0],[x])/\|x\|_1$$ at all $u>0$ for any large enough $T>0$ (proved for the internal distance in Bernoulli percolation in \cite{garet2017continuity}).

Another natural question is to prove that the infinite component in $\mathcal{FI}^{u,T}$ has good isoperimetric bounds (of the type proved in \cite{procaccia2016quenched} for RI). 

\section{Preliminaries on Finitary Random Interlacements}
\label{sec fri def}

In this section, we collect some preliminary results on finitary random interlacements. Most of these results first appear in \cite{bowen2017finitary}. We begin with recalling the formal definition of FRI in \cite{bowen2017finitary}. Consider the lattice $\zd$, for $d \geq 3$. A finite walk on $\zd$ is a nearest-neighbor path $w: \{0,1,\cdots, N\} \rightarrow \zd$, for some $N \in \mathbb{Z}_{+}\cup\{0\}$. $N$ is called the length of the finite walk $w$. Let $\fw$ be the set of trajectories of all finite walks. And note that $\fw$ is a countable set. \par
Throughout this paper, we will use $P$ for probability and $E$ for the corresponding expectation. For $x \in \zd$ and $n \in \mathbb{N}$, let $P_x^n$ be the law of the simple random walk started at $x$ and killed at time $n$. Define
$$
P_x^{(T)} = \bigg(\frac{1}{T+1} \bigg) \sum_{n = 0}^{\infty} \bigg(\frac{T}{T+1} \bigg)^n P_x^n. 
$$
I.e. $P_x^{(T)}$ is the law of a geometrically killed simple random walk started at $x$ with $1/(T+1)$ killing rate. The expected length is $T$. We sometimes call geometrically killed random walk a killed random walk. \par 
For $0 < T < \infty$, let $v^{(T)}$ be the measure on $\fw$ defined by
$$
\vt = \sum_{x \in \zd} \frac{2d}{T+1} P_x^{(T)}. 
$$ 
Note that $\vt$ is a $\sigma$-finite measure. 

\begin{definition}
	\label{def_FRI}
	For $0 < u, T< \infty$, the finitary random interlacements (FRI) point process $\mu$ is a Poisson point process (PPP) on $\fw$ with intensity measure $u \vt$.
\end{definition}

Meanwhile, one may equivalently define $\mathcal{FI}^{u,T}$ constructively as follows:
\begin{definition}
	\label{informal}
	For each vertex $x \in \zd$, define an independent Poisson random variable $N_x$ with parameter $2du/(T+1)$. We start independent $N_x$ geometrically killed random walks from $x$, and each of them has expected length $T$. The FRI can be defined as the point measure on $\fw$ composed of all the geometrically killed random walk trajectories above from all vertices in $\zd$.
\end{definition}
It is easy to see the two definitions above are equivalent: 
\begin{proposition}
	The random point measure defined in Definition \ref{informal} is identically distributed as the Poisson point process defined in Definition \ref{def_FRI}. 
\end{proposition}

\begin{proof}
	The equivalence follows directly from the standard construction of Poisson point process with a $\sigma-$finite intensity measure. See (4.2.1) of \cite{drewitz2014introduction} for example. 
\end{proof}

\begin{remark}
	 The construction in Definition \ref{informal} was informally described in Subsection 1.3.2, \cite{bowen2017finitary}.
\end{remark}
\begin{remark}
Without causing further confusion, we will use $\mathcal{FI}^{u,T}$ to denote both the Poisson point process on $W^{[0,\infty)}$ and the random subgraph of $\BZ^d$ it induces, which has the vertex set the set of vertices visited by $\mathcal{FI}^{u,T}$ and the edge set the set of edges crossed by trajectories in the process $\mathcal{FI}^{u,T}$.
\end{remark}

The rest of this section mainly concerns the distribution of paths within $\mathcal{FI}^{u,T}$ traversing a certain finite subset of $\zd$. Let $K \subset \zd$ be a finite subset. Let $W_K \subset \fw$ be the set of all finite walks that visit $K$ at least once. Define the stopping times
$$
H_K (w) = \inf \{ t \geq 0: w(t) \in K \},
$$ 
and
$$
\tilde{H}_K (w) = \inf \{ t \geq 1: w(t) \in K \}.
$$
For a finite path $w$, we say $H_K (w) = \infty$ if $w$ vanishes before it hits the set $K$. Similar for $\tilde{H}_K (w) = \infty$. Define
$$
W^{(2)} := \{ (a,b) \in \fw \times \fw : a(0) = b(0) \}.
$$
Let $K \subset L \subset \zd$ be finite subsets. For $x \in L \setminus K$, let $\xi_x^{(T)}$ be the measure on $W^{(2)}$ given by
$$
\xi_x^{(T)} (\{(a,b)\}) = 2d \cdot 1_{\tilde{H}_L (a) = \infty} P_x^{(T)} ( \{a\}) 1_{H_K (b) = \infty} P_x^{(T)} ( \{b\}).
$$
Define a measure $Q_{L,K}^{(T)}$ on $W^{(2)}$ by
$$
Q_{L,K}^{(T)} = \sum_{x \in L \setminus K} \xi_x^{(T)}.
$$
Define the concatenation map $\text{Con} : W^{(2)} \rightarrow \fw$ by
$$
\text{Con} (a,b) = \Big( a(len(a)), a(len(a)-1), \cdots , a(0), b(1), \cdots , b(len(b)) \Big).
$$

\begin{proposition} [Proposition $4.1$ in \cite{bowen2017finitary}]
	\label{local}
	For any $0 < u, T< \infty$, let $\mu$ be FRI with parameters $u, T$ and $K \subset L \subset \zd$ be finite subsets. Then $\ind_{W_L \setminus W_K} \mu$ is a PPP with intensity measure $u \cdot Con_{\ast} Q_{L,K}^{(T)} = \ind_{W_L \setminus W_K} uv^{(T)}$, where $Con_{\ast} Q_{L,K}^{(T)} = Q_{L,K}^{(T)} \circ Con^{-1}$ is the push-forward measure.
\end{proposition}

For a finite subset $A \subset \zd$ and $x \in \zd$, we denote the equilibrium measure of $A$ by
$$
e_A (x) := P_x (\tilde{H}_A = \infty) \cdot \ind_{x \in A}.
$$
Define the capacity of $A$ by
$$
\text{cap} (A) := \sum_{x \in \zd} e_A (x).
$$
One can define the random interlacements set $\mathcal{I}^u$, $u>0$ as a random vertex subset of $\zd$ such that for any finite subset $K \subset \zd$, we have $P (\mathcal{I}^u \cap K = \emptyset) = e^{-u \cdot \text{cap} (K)}$. The existence of such random subset is guaranteed in \cite{Sznitman2009Vacant}. By Dynkin's $\pi$-$\lambda$ lemma, there is a unique probability measure on $\{0,1\}^{\zd}$ that samples such random subsets. Random interlacements can also be defined as a Poisson point process of bi-infinite nearest-neighbor trajectories on $\zd$. Readers are referred to \cite{drewitz2014introduction} for a thorough introduction of random interlacements.

Consider the space $\{0,1\}^{\zd}$ with the canonical product $\sigma$-algebra. For $u >0$, let $P^u$ be the unique probability measure on $\{0,1\}^{\zd}$ such that for all finite subset $K \subset \zd$,
$$
P^u (\{w \in \{0,1\}^{\zd}: w(x) = 0, \text{for all } x \in K \}) = e^{-u \cdot \text{cap}(K)},
$$
i.e. $P^u$ is the probability law for random interlacements at level $u$.
For $0 < u, T < \infty$, let $P^{u,T}$ be the probability measure on $\{0,1\}^{\zd}$ such that for all finite subset $K \subset \zd$,
$$
P^{u,T} (\{w \in \{0,1\}^{\zd}: w(x) = 0, \text{for all } x \in K \}) = e^{-2d u \cdot \sum_{x \in K} P_x^{(T)} ( \tilde{H}_K = \infty)},
$$
i.e. $P^{u,T}$ is the law for FRI with parameters $u,T$. The following corollary connects FRI and random interlacements.

\begin{corollary}[Theorem A.$2$ of \cite{bowen2017finitary}]
	\label{local2}
	Let $u, T, \mu$ be as in Proposition \ref{local} and $K \subset \zd$ be a finite subset. Then 
	\begin{enumerate}
	\item $$uv^{(T)} (W_K) = 2d \sum_{x \in K} P_x^{(T)} ( \tilde{H}_K = \infty);$$
	\item $$	\lim_{T \rightarrow \infty} P \big( \mu(W_K) = 0 \big) = e^{-2d u \cdot \text{cap} (K)} = P \big( \mathcal{I}^{2d u} \cap K = \emptyset \big);$$
	\item $P^{u,T}$ converges to $P^{2d u}$ in the weak* topology as $T \rightarrow \infty$ in the space of probability measures on $\{0,1\}^{\zd}$.
	\end{enumerate}
\end{corollary}

\begin{proof}
	Parts $(1)$ and $(2)$ follow from Proposition \ref{local} and the fact that
	$$
	\lim_{T \rightarrow \infty} P_x^{(T)} ( \tilde{H}_K = \infty) = P_x ( \tilde{H}_K = \infty).
	$$
	Part $(3)$ also appears in Theorem A.$2$ of \cite{bowen2017finitary}.
\end{proof}

Let $K \subset \zd$ be a finite subset. Define the killed equilibrium measure by
$$
e_K^{(T)} (x) := (2d) P_x^{(T)} (\tilde{H}_K = \infty) \ind_{x \in K}.
$$
Define the killed capacity by
$$
\text{cap}^{(T)} (K) := \sum_{x \in \zd} e_K^{(T)} (x).
$$
Let
$$
\tilde{e}^{(T)}_K (x) := \frac{e_K^{(T)} (x)}{\text{cap}^{(T)} (K)}
$$
be the normalized equilibrium measure. Let $W^0_K := \{ w \in W_K: w(0) \in K\}$. Define a map
$$
s_K : W_K \ni w \mapsto w^0 \in W^0_K,
$$
where $w^0 = s_K(w)$ is the unique element of $W^0_K$ such that $w^0 (i) = w (H_K(w) + i)$ for all $i \geq 0$ and $len(w^0) = len(w) - H_K(w)$. I.e. we keep the part of the trajectory of $w$ after hitting $K$, and index the trajectory in a way such that the hitting of $K$ occurs at time $0$. If $m(\cdot)$ is a measure supported on $K$, then we define the measure
$$
P_m := \sum_{x \in K} m(x) P_x^{(T)}
$$
on $W_K$, for some $T >0$.

\begin{lemma}
	\label{finite measure 1}
	For $0 < u, T< \infty$, let $\mu$ be FRI with parameters $u, T$ and $K \subset \zd $ be a finite subset. Then $\mu_K = {s_{K}}_{\ast} \mu$ is a PPP on $W_K$ with intensity measure $u \cdot \text{cap}^{(T)} (K) P_{\tilde{e}^{(T)}_K}$. 
\end{lemma}

\begin{proof}
	The proof follows from the Proposition \ref{local} and properties of PPP (see Exercise $4.6$(c) in \cite{drewitz2014introduction}).
\end{proof}

As a consequence of Lemma \ref{finite measure 1}, we have
$$
K \cap \Bigg( \bigcup_{w \in \text{Supp} (\mu_K)} \text{range} (w) \Bigg) = K \cap \Bigg( \bigcup_{w \in \text{Supp} (\mu)} \text{range} (w) \Bigg),
$$
where $K, \mu, \mu_K$ are the same as in Lemma \ref{finite measure 1}.

\begin{lemma}
	\label{finite measure 2}
	Let $N_K$ be a Poisson random variable with parameter $u \cdot \text{cap}^{(T)} (K)$,  and $\{w_j\}_{j \geq 1}$ are i.i.d. killed random walks with distribution $P_{\tilde{e}^{(T)}_K}$ and independent from $N_K$. Then the point measure
	$$
	\tilde{\mu}_K = \sum_{j=1}^{N_K} \delta_{w_j}
	$$
	is a PPP on $W_K$ with intensity measure $u \cdot \text{cap}^{(T)} (K) P_{\tilde{e}^{(T)}_K}$. In particular, $\tilde{\mu}_K$ has the same distribution as $\mu_K$.
\end{lemma}

\begin{proof}
	The proof follows from the construction of PPP (see section $4.2$ in \cite{drewitz2014introduction}) and the merging and thinning property of Poisson distribution.
\end{proof}

\begin{remark}
	A similar result (Corollary 4.2) was proved in \cite{bowen2017finitary}. Here the previous two lemmas are stated in the form better suitable for the later use in this paper. 
\end{remark}

\begin{remark}
The capacity with truncation/killing measure was defined in a
continuous sense in \cite{sznitman2012topics}. It can also be discretized, which gives us the same  $\text{cap}^{(T)} (\cdot)$ as defined in this paper. Thus, similar to \cite{teixeira2009interlacement}, finitary random interlacements may also be equivalently interpreted as random interlacements on a weighted graph with killing measure. This explains why we have representation of finitary random interlacements on compact sets in Lemmas \ref{finite measure 1} and \ref{finite measure 2}.
\end{remark}

\section{Definition of Good Boxes}

\label{sec good def}

Recall the general outline of renormalization argument described in Section \ref{section_intro}. In this section we define the "good" block event in which there is a locally generated large connected cluster in the corresponding ``box". The viability of such event will be proved in the Section \ref{sec good with high probability}. Parts of the definition below are inspired by \cite{rath2011transience}. This also enables us to apply their estimates for regular interlacements in the next section.   

Without loss of generality, we will always assume here the FRI's are constructed according to Definition \ref{informal}. For any $u,T>0$, the FRI $\mathcal{FI}^{u,T}$ is identically distributed as the union of two independent copies of FRI with intensity level $u/2$ and average stopping time $T$, i.e.
$$
\mathcal{FI}^{u,T} = \mathcal{FI}^{u/2,T}_1 \cup \mathcal{FI}^{u/2,T}_2,
$$
where $\mathcal{FI}^{u/2,T}_i$ is the $i$-th copy. For $x \in \zd$ and $R \in \mathbb{Z}_{+}$, let $B(x,R) := x+ [-R,R]^d$ be a box of length $R$ centered at $x$. We write $B(R) = B(0,R)$. Let $\hat{B} (R) := [-64R^2, 64R^2]^d$ be a box in the lattice $\zd$. We define some subboxes in $\hat{B} (R)$. For $0 \leq i \leq 8R$ and $1 \leq j \leq d$, let
$$
x_{i,j} = (-32R^2 + 8Ri) \text{e}_j,
$$
where $\text{e}_j$ is the $j$-th unit vector in $\zd$. Let
$$
b_{i,j} (R) := x_{i,j} + [-R, R]^d \subset \hat{B} (R),
$$
and
$$
\hat{b}_{i,j} (R) := x_{i,j} + [-2R, 2R]^d \subset \hat{B} (R).
$$
For any subset $A \subset \zd$, we define the internal vertex boundary of $A$ by
$$
\partial^{in} A := \{x \in A: \exists y \in \zd \setminus A \text{ such that } |x - y|_1 = 1 \},
$$
and define the external vertex boundary by
$$
\partial^{out} A := \{x \in \zd \setminus A: \exists y \in  A \text{ such that } |x - y|_1 = 1 \}.
$$
Recall the construction of FRI in Definition \ref{informal}. Let $\mathcal{D}_{i}$ be the random subgraph in $\zd$ consisting of all trajectories of killed random walks starting in $B(0, 128R^2)$ in FRI $\mathcal{FI}^{u/2,T}_i$, for $i = 1,2$, and $\mathcal{D} = \mathcal{D}_{1} \cup \mathcal{D}_{2}$. For any subsets $A,B \subset \zd$ where $A$ is connected, let $\mathcal{C} (A, B)$ be the connected component of $A\cup B$ containing $A$. Define the random set
$$
\mathcal{C}_{i,j} (x) := \mathcal{C} \big( x, \hat{b}_{i,j} (R) \cap \mathcal{D}_{1} \big).
$$
For $1 \leq j \leq d$, we define the ``top" half of $\hat{B} (R)$ in the $j$-direction by
$$
\hat{B}^{+}_j (R) = \big\{ x \in \mathbb{R}^d: 0< x_j \leq 64R^2, \text{ and } -64R^2 \leq x_i \leq 64R^2, \text{ if } i \neq j  \big\}, 
$$
and define the ``bottom" half of $\hat{B} (R)$ in the $j$-direction by
$$
\hat{B}^{-}_j (R) = \big\{ x \in \mathbb{R}^d: -64R^2 \leq x_j < 0, \text{ and } -64R^2 \leq x_i \leq 64R^2, \text{ if } i \neq j  \big\}.
$$
Let
$$
A^{+}_j (R) = \big\{ x \in \mathbb{R}^d: 96R^2 \leq x_j \leq 128R^2, \text{ and } -128R^2 \leq x_i \leq 128R^2, \text{ if } i \neq j  \big\}, 
$$
and
$$
A^{-}_j (R) = \big\{ x \in \mathbb{R}^d: -128R^2 \leq x_j \leq -96R^2, \text{ and } -128R^2 \leq x_i \leq 128R^2, \text{ if } i \neq j  \big\}.
$$
\begin{definition}
\label{good def}
We say $\hat{B} (R)$ is \textit{good} if the following conditions hold:
\begin{enumerate}
	\item For all $0 \leq i \leq 8R$ and $1 \leq j \leq d$, let
	$$
	E_{i,j} := \Big\{ x \in b_{i,j} (R) \cap \mathcal{D}_{1} : \text{cap} \big( \mathcal{C}_{i,j} (x) \big) \geq R^{2(d-2)/3} \Big\}.
	$$
	We have $E_{i,j} \neq \emptyset$ for all $i,j$.
	\item For all $0 \leq i < 8R$ and $1 \leq j \leq d$, and for all $x \in E_{i,j}$, and $y \in E_{i+1,j}$,
	$$
	\mathcal{C}_{i+1,j} (y) \cap \mathcal{C}\left(\mathcal{C}_{i,j} (x), \mathcal{D}_{2} \right)\not=\emptyset.
	$$
	I.e., $\mathcal{C}_{i,j} (x)$ and $\mathcal{C}_{i+1,j} (y)$ are connected by $\mathcal{D}_{2}$. 
	\item For all $1 \leq j \leq d$, no geometrically killed random walks starting in $A^{+}_j (R)$ intersect with $\hat{B}^{-}_j (R)$, and no geometrically killed random walks starting in $A^{-}_j (R)$ intersects with $\hat{B}^{+}_j (R)$. 
\end{enumerate}
\end{definition}

\begin{remark}
	\label{independence}
All conditions in Definition \ref{good def} are restrictions on the trajectories of  the killed random walks starting in $B(0, 128R^2)$. This fact is crucial in the renormalization argument in Section \ref{sec renormal}.   
\end{remark}

Now we define the shift of the box $\hat{B} (R)$ in $\zd$. For $x \in \zd$, let
$$
\hat{B}_x (R) = 32R^2 x + \hat{B} (R).
$$
We say that $\hat{B}_x (R)$ is \textit{good} if $\hat{B}(R)$ is a good box in $\mathcal{FI}^{u,T}-32R^2 x$. 

\begin{remark}
	\label{neighbor connect}
Suppose $x$ and $y$ are two neighboring vertices in $\zd$, and both $\hat{B}_x (R)$ and $\hat{B}_y (R)$ are good, then by condition $(3)$ in Definition \ref{good def} the connectivity event in $\hat{B}_x (R) \cap \hat{B}_y (R)$ can be generated only by the random walk paths starting in $B(x, 128R^2) \cap B(y, 128R^2)$, so we have a large connected component crossing $\hat{B}_x (R)$ and $\hat{B}_y (R)$.      
\end{remark}

Now we define a family $\{Y_x : x \in \zd \}$ of $\{0,1\}$-valued random variables given by
\begin{equation}
\label{block coord}
Y_x=\begin{cases}
1, & \text{if $\hat{B}_x (R)$ is good};\\
0, & \text{otherwise}.
\end{cases}
\end{equation}
If there is an infinite open cluster in the lattice $\{Y_x \}_{x \in \zd}$, then by Remark \ref{neighbor connect} there is an infinite open cluster in the underlying original lattice. When $T=R^3$, we will show that $\hat{B} (R)$ is good with high probability for all sufficiently large $R$. Then we will use a renormalization argument to show that there is an infinite cluster in $\mathcal{FI}^{u,R^3}$ almost surely for large $R$. 

\begin{remark}
For simplicity, we will assume $R \in \mathbb{Z}_{+}$ for the rest of this paper. For $R \in \mathbb{R}_{+} \setminus \mathbb{Z}_{+}$, one can replace $R$ and $R^2$ by $\lfloor R \rfloor$ and $\lfloor R \rfloor^2$ respectively in the definition of good boxes, and all results will follow accordingly.
\end{remark}

Throughout the rest of this paper, we denote positive constants by $c, C, c_1,c', \cdots$, and their values can be different from place to place. All constants are dependent on the dimension $d$ by default.

\section{$\hat{B} (R)$ is good with High Probability}

\label{sec good with high probability}

In this section, we prove that $\hat{B} (R)$ is good with high probability. I.e., 

\begin{theorem}
	\label{good}
Consider the FRI $\mathcal{FI}^{u,R^3}$. For all $u >0$, we have
$$
\lim_{R \rightarrow \infty} P \big( Y_0 = 1 \big) = 1.
$$
\end{theorem}

To show Theorem \ref{good}, we will consider the following weaker version of conditions $(1)$ and $(2)$ in Definition \ref{good def}:
\begin{enumerate}[(1*)]
	\item For all $0 \leq i \leq 8R$ and $1 \leq j \leq d$, let
	$$
	 \mathcal{\tilde C}_{i,j} (x) := \mathcal{C} \big( x, \hat{b}_{i,j} (R) \cap \mathcal{FI}^{u,T}_1  \big).
	$$
	and
	$$
	\tilde{E}_{i,j} := \Big\{ x \in b_{i,j} (R) \cap \mathcal{FI}^{u,T}_1 : \text{cap} \big( \mathcal{\tilde C}_{i,j} (x) \big) \geq R^{2(d-2)/3} \Big\}.
	$$
	We have $\tilde{E}_{i,j} \neq \emptyset$ for all $i,j$.
	\item For all $0 \leq i < 8R$ and $1 \leq j \leq d$, and for all $x \in \tilde{E}_{i,j}$, and $y \in \tilde{E}_{i+1,j}$,
	$$
	\tilde{\mathcal{C}}_{i+1,j} (y) \cap \mathcal{C}\left(\tilde{\mathcal{C}}_{i,j} (x), \mathcal{FI}^{u,T}_2 \right)\not=\emptyset.
	$$	
\end{enumerate}

We first prove that condition $(1^{*})$ and $(2^{*})$ occur with high probability. Then we show that no killed random walk starting in $\zd \setminus B(128R^2)$ will reach $\hat{B} (R)$ with high probability. Combining these we know condition $(1)$ and $(2)$ in Definition \ref{good def} occur with high probability. We will show condition $(3)$ occurs with high probability separately in Lemma \ref{cond 3}.

We will often use the following large deviation bound for Poisson distributions.

\begin{lemma}[Equation $2.11$ in \cite{rath2011transience}]
	\label{poi concent}
	If $X$ is a Poisson distribution with parameter $\lambda$, then 
	$$
	P \big( \lambda /2 \leq X \leq 2 \lambda) \geq 1 - 2e^{- \lambda /10}.
	$$
\end{lemma}

\subsection{Coupling of FRI and RI}
\label{coupling}
In this subsection we introduce a coupling of FRI and RI that is crucial in the proof of Lemma \ref{large cap}. Let $K \subset \zd$ be a finite subset, and let $u,T>0$. For any points $x \in K$, let $N_{x,u}$ be i.i.d. Poisson random variables with parameter $u$. Let $\{Y_{x,T}^{(l,i)} +1 \}_{i=1}^{\infty}$ and $\{Y_{x,T}^{(r,i)} +1 \}_{i=1}^{\infty}$ be i.i.d. geometric random variables with parameter $1/(T+1)$. Moreover, for $i \in \mathbb{Z}_{+}$, let $\{S_{n,x}^{(l,i)}\}_{n=0}^{\infty}$ and $\{S_{n,x}^{(r,i)}\}_{n=0}^{\infty}$ be independent copies of simple random walks starting at $x$. Now we can construct a random point measure $\mathcal{I}^T (u,K)$ on $W^{[0,\infty)}$ as follows: for each $x \in K$ and $1 \leq i \leq N_{x,u}$, if
$$
\{S_{n,x}^{(l,i)}\}_{n=1}^{Y_{x,T}^{(l,i)}} \cap K = \emptyset,
$$
we add a delta measure on
$$
 \{S_{n,x}^{(r,i)}\}_{n=0}^{Y_{x,T}^{(r,i)}}
$$
in $\mathcal{I}^T (u,K)$.

The following lemma is a consequence of Lemma \ref{finite measure 2}. Let $\mu_K = \sum_{j=1}^{N_K} \delta_{w_j}$ be the restriction of FRI Poisson point measure on $K$, where $N_K$ is a Poisson random variable with parameter $u \cdot \text{cap}^{(T)} (K)$, and $\{w_j\}_{j \geq 1}$ are i.i.d. killed random walks with distribution $P_{\tilde{e}^{(T)}_K}$ and independent from $N_K$. 

\begin{lemma}
	\label{FRI couple}
	$\mathcal{I}^T (u,K)$ is identically distributed as $\mu_K$. 
\end{lemma}

\begin{proof}
	Notice that if we fix $x \in K$ and $1 \leq i \leq N_{x,u}$, then
	$$
	P \bigg( \{S_{n,x}^{(l,i)}\}_{n=1}^{Y_{x,T}^{(l,i)}} \cap K = \emptyset \bigg) = P_x^{(T)} (\tilde{H}_K = \infty ) = e_K^{(T)} (x).
	$$
	By Lemma \ref{finite measure 2}, $\mu_K$ is a PPP with intensity measure $u \cdot \text{cap}^{(T)} (K) P_{\tilde{e}^{(T)}_K}$, and by definition
	$$
	e_K^{(T)} (x) = \text{cap}^{(T)} (K) \tilde{e}^{(T)}_K.
	$$
	The result follows from the thinning property of Poisson distributions.
\end{proof}

Consider those trajectories in $\mathcal{I}^T (u,K)$ with length larger or equal to a fixed number $T_0 >0$. We define the random point measure $\hat{\mathcal{I}}^{T,T_0} (u,K)$ as follows: for each $x \in K$ and $1 \leq i \leq N_{x,u}$, if
$$
Y_{x,T}^{(r,i)} \geq T_0,
$$
and
$$
\{S_{n,x}^{(l,i)}\}_{n=1}^{Y_{x,T}^{(l,i)}} \cap K = \emptyset,
$$
we add a delta measure on
$$
\{S_{n,x}^{(r,i)}\}_{n=0}^{Y_{x,T}^{(r,i)}}
$$
in $\hat{\mathcal{I}}^{T,T_0} (u,K)$. Note that by definition $\hat{\mathcal{I}}^{T,T_0} (u,K) \subset \mathcal{I}^T (u,K)$. Here we say $\CI_1\subset \CI_2$ if all edges open in the support of $\CI_1$ is also open in support of $\CI_2$. \par

Now we construct a third random point measure $\bar{\mathcal{I}}^{T, T_0} (u,K)$ which is identically distributed as the collection of all trajectories within a RI traversing $K$, and we also define a $\tilde{\mathcal{I}}^{T,T_0} (u,K) \subset \bar{\mathcal{I}}^{T,T_0} (u,K)$ when all trajectories in $\bar{\mathcal{I}}^{T,T_0} (u,K)$ are truncated at a fixed time $T_0$. For each $x \in K$ and $1 \leq i \leq N_{x,u}$, if
$$
Y_{x,T}^{(r,i)} \geq T_0,
$$
and
$$
\{S_{n,x}^{(l,i)}\}_{n=1}^{\infty} \cap K = \emptyset,
$$
we add a delta measure on
$$
\{S_{n,x}^{(r,i)}\}_{n=0}^{\infty}
$$
in $\bar{\mathcal{I}}^{T,T_0} (u,K)$ and we add a delta measure on 
$$
\{S_{n,x}^{(r,i)}\}_{n=0}^{T_0}
$$
in $\tilde{\mathcal{I}}^{T,T_0} (u,K)$. By definition $\tilde{\mathcal{I}}^{T,T_0} (u,K) \subset \bar{\mathcal{I}}^{T, T_0} (u,K)$ for any $T, T_0 > 0$. If $T_0 = 0$, $\bar{\mathcal{I}}^{T, 0} (u,K)$ is identically distributed as the set of all trajectories in $\mathcal{I}^u$ traversing $K$ but not including the backward parts before they enter $K$ for the first time. We write $ \bar{\mathcal{I}}^{T} (u,K) := \bar{\mathcal{I}}^{T, 0} (u,K)$.

\begin{lemma}
	\label{ri coupling}
	Let $Y+1$ be a geometric random variable with parameter $1/(T+1)$ independent from everything else, and $q = q(T, T_0) := P (Y \geq T_0)$.  Let $\bar{\mu}_K$ be the restriction of PPP for random interlacements at level $uq$ on the set $K$, i.e. $\bar{\mu}_K = \sum_{j=1}^{\bar{N}_K} \delta_{\bar{w}_j}$ is a random point measure, where $\bar{N}_K$ is a Poisson random variable with parameter $uq \cdot \text{cap} (K)$, and $\{\bar{w}_j\}_{j \geq 1}$ are i.i.d. simple random walks with distribution $P_{e_K}$ and independent from $\bar{N}_K$. Then $\bar{\mathcal{I}}^{T, T_0} (u,K)$ is identically distributed as $\bar{\mu}_K = \sum_{j=1}^{\bar{N}_K} \delta_{\bar{w}_j}$.
\end{lemma}

\begin{proof}
	This is similar to the proof of Lemma \ref{FRI couple}. For $x \in \partial^{in} K$, 
	$$
	P \bigg( \{S_{n,x}^{(l,i)}\}_{n=1}^{\infty} \cap K = \emptyset \bigg) = P_x (\tilde{H}_K = \infty) = e_K (x).
	$$
	Note that for all $x \in K \setminus \partial^{in} K$,
	$$
	P \bigg( \{S_{n,x}^{(l,i)}\}_{n=1}^{\infty} \cap K = \emptyset \bigg) = 0.
	$$
	The result again follows from the thinning property of Poisson distributions.
\end{proof}

\subsection{Facts about capacity}
We often use the following facts about capacity (or killed one) in our proof.

\begin{lemma}[Proposition $6.5.2$ in \cite{lawler2010random}]
	\label{cap esti}
	There are constants $c_1, c_2 >0$ such that for all $R>0$,
	$$
	c_1 R^{d-2} \leq \text{cap} \big( B(R) \big) \leq c_2 R^{d-2}.
	$$
\end{lemma}



\begin{lemma}[Monotonicity of Capacity; Exercise $1.15$ in \cite{drewitz2014introduction}]
For any finite sets $E_1 \subset E_2 \subset \zd$,
$$
\text{cap} (E_1) \leq \text{cap} (E_2).
$$	
\end{lemma}

\subsection{Condition $(1^{*})$}
\label{1star}

Similar to \cite{rath2011transience}, we may write 
$$
\mathcal{FI}^{u/2,T}_1=\bigcup_{k=1}^{d-2} \mathcal{FI}^{u/(2d-4),T}_{1,k},
$$
where $\mathcal{FI}^{u/(2d-4),T}_{1,k}$ are i.i.d. copies of finitary interlacements with intensity level $u/(2d-4)$ and average stopping time $T$. By translation invariance, one may without loss of generality prove the desired result for $i=4R$ and $j=1$. This case, we have $x_{4R,1}=0$, $b_{4R,1}(R)=B(R)$, and $\hat b_{4R,1}(R)=B(2R)$. 

To begin with, let us consider the following random variable 
$$
N^{(1)}_{4R,1}:=\left|\left\{x\in B(R), cap\left(\CC\Big(x,\mathcal{FI}^{u/(2d-4),R^3}_{1,1}\cap B(R+R^{0.9})\Big)\right)>c_0 R^{0.7}\right\} \right|
$$
and event $A^{(1)}_{4R,1}=\{ N^{(1)}_{4R,1}\ge 1\}$, where $c_0>0$ is the constant in Lemma 6, \cite{rath2011transience}, which is independent to $R$. We first prove that
\begin{lemma}
	\label{large cap}
There is a constant $c = c(u)>0$ such that for all sufficiently large $R$, $P(A^{(1)}_{4R,1})\ge 1-\exp(-c R^{d-2})$. 
\end{lemma}
\begin{proof}
Note that $N^{(1)}_{4R,1}$ is determined by trajectories within $\mathcal{FI}^{u/(2d-4),R^3}_{1,1}$ traversing $B(R)$, which can be sampled according to Subsection \ref{coupling}. Recalling the notations used there, we have $N^{(1)}_{4R,1}$ stochastically dominates the random variable $\hat N^{(1)}_{4R,1}$ where 
$$
\begin{aligned}
\hat N^{(1)}_{4R,1}=\Big| \Big\{&(x,i)\in \partial^{in} B(R)\times \BZ^+, \ s.t. \ i\le N_{x, u/(2d-4)}, \ \{S_{n,x}^{(l,i)}\}_{n=1}^{\infty}\cap B(R)=\emptyset, \\
& Y_{x,R^3}^{r,i}\ge R^{1.6}, \ \{S_{n,x}^{(r,i)}\}_{n=1}^{R^{1.6}}\subset x+B(R^{0.9}), \ cap\left(\{S_{n,x}^{(r,i)}\}_{n=1}^{R^{1.6}} \right)>c_0 R^{0.7}\Big\} \Big|,
\end{aligned}
$$
and $c_0$ is the same constant in the definition of $N^{(1)}_{4R,1}$. Note that for each $(x,i)$, the events 
\begin{align*}
	&\{i\le N_{x, u/(2d-4)}\},\\
	&\left\{ \{S_{n,x}^{(l,i)}\}_{n=1}^{\infty}\cap B(R)=\emptyset\right\},\\
	&\left\{ Y_{x,R^3}^{r,i}\ge R^{1.6} \right\},\\
	&\left\{ \{S_{n,x}^{(r,i)}\}_{n=1}^{R^{1.6}}\subset x+B(R^{0.9}), \ cap\left(\{S_{n,x}^{(r,i)}\}_{n=1}^{R^{1.6}} \right)>c_0 R^{0.7} \right\}
\end{align*}
are independent to each other. At the same time 
$$
P\left( \{S_{n,x}^{(l,i)}\}_{n=1}^{\infty}\cap B(R)=\emptyset\right)=e_{B(R)}(x)
$$
while 
$$
P\left(Y_{x,R^3}^{r,i}\ge R^{1.6}, \ \{S_{n,x}^{(r,i)}\}_{n=1}^{R^{1.6}}\subset x+B(R^{0.9}), \ cap\left(\{S_{n,x}^{(r,i)}\}_{n=1}^{R^{1.6}} \right)>c_0 R^{0.7}\right)=q_1(R)>1/2 
$$
for all sufficiently large $R$. The last inequality is derived from 
\begin{enumerate}[(1)]
	\item The PMF estimate of geometric random variable $Y_{x,R^3}^{r,i}$. 
	\item Hoeffding’s inequality. 
	\item Lemma 6, \cite{rath2011transience} with $T_1=R^{1.6}$ and $\ep=1/8$.  
\end{enumerate}
Thus we have 
$$
\hat N^{(1)}_{4R,1}\sim Poisson\Big(q_1(R)cap(B(R))u/(2d-4) \Big)
$$
and the desired result follows from Lemma \ref{poi concent} and Lemma \ref{cap esti}. 

\end{proof} 

Given the event $A^{(1)}_{4R,1}$, one may sample a point uniformly at random from the random subset 
$$
S_{4R,1}=\left\{x\in B(R), cap\left(\CC\Big(x,\mathcal{FI}^{u/(2d-4),R^3}_{1,1}\cap B(R+R^{0.9})\Big)\right)>c_0 R^{0.7}\right\}
$$
and denote it by $x^{(1)}_{4R,1}$. Moreover, for the random subset 
$$
Com^{(1)}_{4R,1}=\CC\Big(x^{(1)}_{4R,1},\mathcal{FI}^{u/(2d-4),R^3}_{1,1}\cap B(R+R^{0.9})\Big)
$$
by definition we have 
$$
cap\left(Com^{(1)}_{4R,1} \right)> c_0 R^{0.7}. 
$$
Now for any $k=2,3,\cdots, d-2$ may define 
$$
Com^{(k)}_{4R,1}=\CC\left(Com^{(k-1)}_{4R,1},\mathcal{FI}^{u/(2d-4),R^3}_{1,k}\cap B(R+kR^{0.9})\right)
$$
together with the event 
$$
A^{(k)}_{4R,1}=\left\{cap\big(Com^{(k)}_{4R,1}\big)>c_0^k R^{0.7 k}\right\}. 
$$
Note that for any $k=2,3,\cdots, d-2$, $Com^{(k-1)}_{4R,1}$ is measurable with respect to 
$$
\sigma_{k-1}=\sigma\left(\mathcal{FI}^{u/(2d-4),R^3}_{1,1}, \mathcal{FI}^{u/(2d-4),R^3}_{1,2},\cdots, \mathcal{FI}^{u/(2d-4),R^3}_{1,k-1} \right)
$$
which is independent to $\mathcal{FI}^{u/(2d-4),R^3}_{1,k}$. Let $\CC_{0}^{(k-1)}$ be a connected component within $B(R+(k-1)R^{0.9})$ such that
$$
cap(\CC_{0}^{(k-1)})>c_0^{k-1} R^{0.7 (k-1)}.
$$
Given $Com^{(k-1)}_{4R,1}=\CC_{0}^{(k-1)}$, the distribution of $Com^{(k)}_{4R,1}$ is determined by the configuration of trajectories in $\mathcal{FI}^{u/(2d-4),R^3}_{1,k}$ traversing $\CC_{0}^{(k-1)}$, which can again be sampled according to Subsection 5.1: 
\begin{itemize}
	\item For each $x\in \CC_{0}^{(k-1)}$, let $N^{(k)}_{x, u/(2d-4)}$ be i.i.d. Poisson random variables independent to $\sigma_{k-1}$ with intensity $u/(2d-4)$.  
	\item For each $x\in \CC_{0}^{(k-1)}$, and positive integer $i$, let $\{S_{n,x}^{(l,i,k)}\}_{n=1}^{\infty}$ and $\{S_{n,x}^{(r,i,k)}\}_{n=1}^{\infty}$ be independent simple random walks starting from $x$. 
	\item For each $x\in \CC_{0}^{(k-1)}$, and positive integer $i$, let $Y_{x,R^3}^{r,i,k}$ and $Y_{x,R^3}^{l,i,k}$ be independent geometric random variables with parameter $p=1/(1+R^3)$.
\end{itemize}
Recalling the construction in Subsection \ref{coupling}, one has 
$$
\begin{aligned}
&P\left(A^{(k)}_{4R,1}\big| Com^{(k-1)}_{4R,1}=\CC_{0}^{(k-1)}\right)\\
\ge &P\left(cap\left(\bigcup_{(x,i)\in I^{(k-1)}_{4R,1}} \{S_{n,x}^{(r,i,k)}\}_{n=1}^{R^{1.6}}\right)>c_0^k R^{0.7 k}, \ \{S_{n,x}^{(r,i,k)}\}_{n=1}^{R^{1.6}}\subset x+B(R^{0.9}), \ \forall (x,i)\in  I^{(k-1)}_{4R,1} \right)
\end{aligned}
$$
where
$$
I^{(k-1)}_{4R,1}=\Big\{(x,i)\in \partial^{in} \CC_{0}^{(k-1)}\times \BZ^+, \ s.t. \ i\le N^{(k)}_{x, u/(2d-4)}, \ \{S_{n,x}^{(l,i,k)}\}_{n=1}^{\infty}\cap \CC_{0}^{(k-1)} =\emptyset, \
 Y_{x,R^3}^{r,i,k}\ge R^{1.6} \Big\}. 
$$
Note that the set $I^{(k-1)}_{4R,1}$ has the same law as the set of trajectories in $$\bar{\mathcal{I}}^{R^3,R^{1.6}}(u/(2d-4),\CC_{0}^{(k-1)}),$$ By Lemma \ref{ri coupling}, for all sufficiently large $R$, $I^{(k-1)}_{4R,1}$ has the same law as random  interlacements at level $uq/(2d-4)$ hitting $\CC_{0}^{(k-1)}$ for some $q > 1/2$. Recall that $\CC_{0}^{(k-1)}$ is a fixed set. By Lemmas $7$ and Lemma $8$ (with $s=1$ there) in \cite{rath2011transience},
$$
P\left(A^{(k)}_{4R,1}\big| Com^{(k-1)}_{4R,1}=\CC_{0}^{(k-1)}\right)\ge 1-\exp(-R^{1/17})
$$
for all sufficiently large $R$. Thus we have proved that 
\begin{equation}
\label{condition_1}
P(\tilde{E}_{4R,1}\not=\emptyset)\ge P\left(\bigcap_{k=1}^{d-2}A^{(k)}_{4R,1} \right)
\ge 1-\exp(-R^{1/18})
\end{equation}
for all sufficiently large $R$. 

\subsection{Condition $(2^{*})$}
\label{2star}

Again, Condition ($2^{*}$) can be without loss of generality checked for $b_{4R,1}(R)$ and $b_{4R+1,1}(R)$. One may follow a similar argument as Subsection \ref{1star} to check Condition ($2^{*}$). To be precise, one can pick any two points $x_0,x_1$ from $\tilde{E}_{4R,1}$ and $\tilde{E}_{4R+1,1}$. Then we can look at the paths in $\mathcal{FI}_{2}^{u/2,R^3}$ (which is independent to $\mathcal{FI}_{1}^{u/2,R^3}$) traversing $\tilde{\mathcal{C}}_{4R,1}(x_0)$. We keep only those whose backward part never returning to $\tilde{\mathcal{C}}_{4R,1}(x_0)$ while the forward part is not truncated until the $R^{2.5}$th step. Then one can apply Lemma 11 and 12 in \cite{rath2011transience} for intensity $u/4$ to prove that with stretch exponentially high probability, at least one of the paths we kept in the procedure above has to intersect with $\tilde{\mathcal{C}}_{4R+1,1}(x_1)$ before they exit $B(4Re_1, CR)$, where $C$ is the same constant as in Lemma 11 of \cite{rath2011transience}. 

However, since for the finitary random interlacements, one can only guarantee that the first $R^{2.5}$ steps in the forward paths we keep are within $\mathcal{FI}_{2}^{u/2,R^3}$. So the only extra estimate needed is the following lower bound on the first exiting time of $B(CR)$. 
\begin{lemma}
	\label{connect before escape}
	There is a $c>0$ independent to $R$ such that 
	$$
	P_0(H_{\partial^{out}B(CR)}>R^{2.5})<\exp(-c R^{0.5}). 
	$$
\end{lemma}

\begin{proof}
By central limit theorem/invariance principle, there is a constant $c>0$ such that  
\begin{equation}
\label{CLT_escape}
\sup_{x\in B(CR)}P_x(H_{\partial^{out}B(CR)}>R^{2})\le P_0(H_{\partial^{out}B(2CR)}>R^{2})\le 1-c<1. 
\end{equation}
Then for each $i=1,2,\cdots, [R^{0.5}]$, consider event
$$
Es_i=\{H_{\partial^{out}B(CR)}>i*R^{2}\}.
$$
Then by \eqref{CLT_escape} and Markov property we have
$$
P_0(Es_1)\le 1-c, 
$$
and
$$
P_0(Es_{i+1}|Es_{i})\le \sup_{x\in B(CR)}P_x(H_{\partial^{out}B(CR)}>R^{2})\le 1-c,
$$
for all $i\ge 1$. Thus
$$
P_0(H_{\partial^{out}B(CR)}>R^{2.5})\le P_0(Es_{\lfloor R^{0.5}\rfloor}) \le (1-c)^{\lfloor R^{0.5}\rfloor}<\exp(-c R^{0.5}). 
$$
\end{proof}

\begin{remark}
An alternative argument following (2.9) of \cite{MR3770879} derives a slightly weaker result, but also suitable for the use here. 
\end{remark}

Suppose $U,V \subset B(CR)$. Taking $q=1/2$ in Lemma \ref{ri coupling}, we know that for all sufficiently large $R$, the set of trajectories in $\bar{\mathcal{I}}^{R^3, R^{2.5}}(u/2, U)$ stochastically dominates the ones in $\mathcal{I}^{u/4}$ hitting $U$. Combining this fact with Lemma \ref{connect before escape} and Lemmas $11$, $12$ of \cite{rath2011transience}, we have
\begin{equation}
\label{connect 2 sets}
P \Bigg( U \xleftrightarrow{\tilde{\mathcal{I}}^{R^3, R^{2.5}}(u/2, U) \cap B(CR)} V \Bigg) \geq 1 - c_1 e^{-c_2 \min\{R^{0.5}, R^{2-d} \text{cap} (U) \text{cap} (V)\}}.
\end{equation}
Replacing $U$ and $V$ by $\tilde{\mathcal{C}}_{4R,1}(x_0)$ and $\tilde{\mathcal{C}}_{4R+1,1}(x_1)$ in (\ref{connect 2 sets}), we prove Condition $(2^{*})$.

\subsection{Condition $(1)$ and $(2)$}
We recall the construction of FRI in Definition \ref{informal}. We first show that with high probability no killed random walks of $\mathcal{FI}^{u, R^3}$ starting in $\zd \setminus B(128R^2)$ intersect with $\hat{B} (R)$. Define the event
$$
G(u, R) := \Big\{ \text{No killed random walks of } \mathcal{FI}^{u, R^3} \text{ starting in } \zd \setminus B(128R^2) \text{ reach } \hat{B} (R) \Big\}.
$$

\begin{lemma}
\label{no outside}
For all $u >0$, we have
$$
\lim_{R \rightarrow \infty} P \big( G(u, R) \big) = 1.
$$
\end{lemma}

\begin{proof}
We first fix $u >0$ and $R >0$. We define a sequence of subsets $\{A(m, R)\}_{m=1}^{\infty}$ of $\zd$. Let
$$
A(1, R) := B \big( (128 + 64)R^2 \big) \setminus B(128R^2),
$$
and for all $m > 1$,
$$
A(m, R) := B \big( (128 + 64m)R^2 \big) \setminus B \big( (128 + 64(m-1) )R^2 \big) 
$$
Note that $\{A(m, R)\}_{m=1}^{\infty}$ are pairwise disjoint, and
$$
\zd = \bigg( \hat{B} (R) \cup \bigcup_{m=1}^{\infty} A(m, R) \bigg). 
$$
Let $x \in A(m, R) \cap \zd$ for some $m \geq 1$. Recall the construction of FRI in Definition \ref{informal}. Let $N_x$ be the number of killed random walks starting at $x$, so $N_x$ is a Poisson distribution with parameter $2du/(R^3+1)$. By Markov inequality, for all sufficiently large $R$, 
$$
P \bigg( N_x > \frac{2dumR^4}{R^3+1} \bigg) \leq E[e^{N_x}] e^{-2dumR^4/(R^3+1)} \leq c_1 e^{-c_2 mR},
$$ 
for some constants $c_1 (u), c_2 (u) >0$. We also need to estimate the probability that a killed random walk escape from a big box. If $Y$ is a geometric random variable with parameter $1/(R^3+1)$, then for all sufficiently large $R$ and for all $m$,
\begin{equation}
\label{geo tail}
P(Y > mR^{7/2}) \leq e^{-cmR^{1/2}},
\end{equation}
for some $c>0$ independent of $R$. By Azuma's inequality and the tail estimate of geometric distribution in (\ref{geo tail}), for all sufficiently large $R$ and for all $x \in A(m, R) \cap \zd$,
$$
P_x^{(R^3)} \big( H_{\hat{B} (R)} < \infty \big) \leq e^{- c_3 m  R^{1/2}}.
$$
Note that the number of vertices in $A(m, R)$ is bounded above by $c_4 m^d R^{2d}$, for some $c_4 >0$. So by union bound,
$$
P \big( G(u, R)^{\text{c}} \big) \leq \sum_{m=1}^{\infty} \bigg( c_4 m^d R^{2d} c_1 e^{-c_2 mR} + c_4 m^d R^{2d} \frac{2dumR^4}{R^3+1} e^{- c_3 m  R^{1/2}} \bigg),    
$$
for all sufficiently large $R$. Let
$$
S(R) := \sum_{m=1}^{\infty} \bigg( c_4 m^d R^{2d} c_1 e^{-c_2 mR} + c_4 m^d R^{2d} \frac{2dumR^4}{R^3+1} e^{- c_3 m  R^{1/2}} \bigg).
$$
Note that the sum $S(R)$ converges for all $R >0$, and
$$
S(R) \xrightarrow{R \rightarrow \infty} 0.
$$
Therefore,
$$
P \big( G(u, R)^{\text{c}} \big) \xrightarrow{R \rightarrow \infty} 0. 
$$ 

\end{proof}

\begin{lemma}
Let $u >0$. Consider the FRI $\mathcal{FI}^{u, R^3}$. Then
$$
\lim_{R \rightarrow \infty} P \Big( \text{Conditions } (1) \text{ and } (2) \text{ are satisfied }\Big) =1.
$$
\end{lemma}

\begin{proof}
The result follows by the discussions in Subsections \ref{1star} and \ref{2star}, and Lemma \ref{no outside}.
\end{proof}

\subsection{Condition $(3)$}
By translation invariance and symmetry, it suffices to show the following lemma.

\begin{lemma}
\label{cond 3}
Let $u >0$, then there are constants $c(u), C(u) >0$ such that for all sufficiently large $R >0$, we have
$$
P \Big( \exists \text{ a killed random walk starting in } A^{+}_1 (R) \text{ reach } \hat{B}^{-}_1 (R) \Big) \leq c R^{2d+1} e^{- CR^{1/2}}.
$$ 
\end{lemma}

\begin{proof}
One can easily adapt the calculations in the proof of Lemma \ref{no outside}. The result follows from Definition \ref{informal}, and tail estimates of geometric and Poisson distributions, and Azuma's inequality. 
\end{proof}

\section{Renormalization and proof of Theorem \ref{main1}}
\label{sec renormal}
Recall the family $\{Y_x\}_{x \in \zd}$ of $\{0,1\}$-valued random variables defined in (\ref{block coord}). In this section, we show that $\{Y_x\}$ stochastically dominates an i.i.d. supercritical site percolation when $R$ is sufficiently large and thus it has an infinite open cluster almost surely.

\begin{remark}
	Note that $\{Y_x\}_{x \in \zd}$ themselves form a finitely dependent percolation, and that the probability that each edge is open is high enough. An alternative ``block construction" approach according to Durrett and Griffeath, \cite{MR682796} can also give us the desired result. 	
\end{remark}

\begin{lemma}
	\label{dom}
For any $u > 0$ and for all $R > 0$ that is sufficiently large (depending on $u$), the random field $\{Y_x\}_{x \in \zd}$ generated by $\mathcal{FI}^{u,R^3}$ stochastically dominates an i.i.d. site percolation $\{Z_x\}_{x \in \zd}$ such that $P (Z_0 = 1) > p_c (\zd)$, where $p_c (\zd)$ is the critical probability of site percolation on $\zd$. 
\end{lemma}

\begin{proof}
By the definition of good boxes in Section \ref{sec good def} and Remark \ref{independence}, the random field $\{Y_x\}_{x \in \zd}$ is $9$-dependent. The stochastic domination over an i.i.d supercritical site percolation follows from the domination by product measures result by Liggett, Schonmann, and Stacey \cite{Liggett1997Domination} (or Theorem $7.65$ in \cite{MR1707339}) and Theorem \ref{good}.
\end{proof}

\begin{corollary}
	\label{inf cluster}
For any $u > 0$ and for all $R > 0$ that is sufficiently large (depending on $u$), $\mathcal{FI}^{u,R^3}$ has an infinite cluster  almost surely.
\end{corollary}

\begin{proof}
We can choose the same $R$ as in Lemma \ref{dom}. By the definition of good boxes and Remark \ref{neighbor connect}, $\mathcal{FI}^{u,R^3}$ has an infinite cluster if $\{Y_x\}_{x \in \zd}$ has one.
\end{proof}

Now back to the proof of Theorem \ref{main1}, for any $u>0$ and sufficiently large $T$, one may let $R=\lfloor T^{1/3}\rfloor$ and the proof is complete. \qed

\section{Uniqueness of Infinite Cluster}
\label{sec unique}

We have shown that the FRI $\mathcal{FI}^{u,R^3}$ has an infinite cluster almost surely if $R > R_0 (u)$, for some $R_0 (u)> 0$. In this section, we show that the infinite cluster of $\mathcal{FI}^{u,R^3}$ is unique almost surely. Let $x \in \zd$, we define the canonical lattice shift
$$
T_x: \{0,1\}^{\zd} \rightarrow \{0,1\}^{\zd}
$$
by $\big( T_x (\xi) \big) (y) = \xi(y+x)$, for any $\xi \in \{0,1\}^{\zd}$ and $y \in \zd$. We will first show that FRI is ergodic with respect to lattice shifts.

\begin{lemma}
	\label{measure preserving}
	Let $P^{u,T}$ be the probability law for $\mathcal{FI}^{u,T}$ defined in Section \ref{sec fri def}. For any $x \in \zd$ and any $u,T>0$, the map $T_x$ preserves $P^{u,T}$.
\end{lemma} 

\begin{proof}
	Fix $x \in \zd$. By Dynkin's $\pi$-$\lambda$ Lemma, it suffices to show that for any finite subset $K \subset \zd$,
	$$
	P \big( \mathcal{FI}^{u,T} \cap (K-x) = \emptyset \big) = P \big( \mathcal{FI}^{u,T} \cap K = \emptyset \big) = e^{-u \cdot \text{cap}^{(T)}(K)}.
	$$
	Note that
	$$
	P \big( \mathcal{FI}^{u,T} \cap (K-x) = \emptyset \big) = e^{-u \cdot \text{cap}^{(T)}(K-x)} = e^{-u \cdot \text{cap}^{(T)}(K)}.
	$$
	The proof is complete.
\end{proof}

Let $x \in \zd$, define the evaluation map
$$
\Phi_x : \{0,1\}^{\zd} \rightarrow \{0,1\}
$$
by $\Phi_x (\xi) = \xi (x)$. We write $\sigma (\cdot)$ for the product $\sigma$-algebra generated by a set or the $\sigma$-algebra generated by a set of functions. The following lemma is a classical approximation result.

\begin{lemma}
	\label{approx}
	Let $\big( \{0,1\}^{\zd}, \sigma (\{0,1\}^{\zd}), Q \big)$ be a probability space, and let $B \in \sigma (\{0,1\}^{\zd})$, then for any $\ep > 0$, there is a finite subset $K \subset \zd$ and $B_{\ep} \in \sigma (\Phi_x : x \in K)$ such that
	$$
	Q \big( B \triangle B_{\ep} \big) \leq \ep.
	$$
\end{lemma} 

We need one more auxiliary lemma.

\begin{lemma}
	\label{cap cal}
	Let $K \subset \zd$ be a finite subset, and $K_1 \subset K$, and $K_0 = K \setminus K_1$. Then for all $u,T>0$,
	$$
	P \big( \mathcal{FI}^{u,T} \cap K = K_1 \big) = \sum_{ K' \subset K_1 } (-1)^{|K'|} e^{-u \cdot \text{cap}^{(T)} (K' \cup K_0)}.
	$$
\end{lemma}

\begin{proof}
	This follows from inclusion-exclusion formula (see Equation $2.1.3$ of \cite{drewitz2014introduction} for a similar result in RI).
\end{proof}

\begin{proposition}
	\label{ergodic}
	For any $u,T>0$ and any $0 \neq x \in \zd$, the measure preserving map $T_x$ is ergodic with respect to $P^{u,T}$.
\end{proposition}

\begin{proof}
	One can easily adapt the proof of ergodicity for random interlacements, e.g. see Theorem $2.1$ of \cite{Sznitman2009Vacant}. 

\end{proof}

\begin{theorem}
	\label{unique}
	For any $u>0$ and for all sufficiently large $R>0$ (depending on $u$), $\mathcal{FI}^{u,R^3}$ has a unique infinite open cluster almost surely.
\end{theorem}

\begin{proof}
	We adapt the proof of uniqueness in percolation model by Burton and Keane \cite{Burton1989Density} (see Theorem $8.1$ in \cite{MR1707339} and Theorem $12.2$ in \cite{haggstrom2006uniqueness}). Fix $u >0$. Let $N$ be the number of infinite open clusters in $\mathcal{FI}^{u,R^3}$. Since $N$ is translation-invariant, $N$ is constant almost surely by Proposition \ref{ergodic}. By Corollary \ref{inf cluster}, there is a $R_0 (u) >0$ such that for all $R > R_0$, $\mathcal{FI}^{u,R^3}$ has an infinite open cluster almost surely. We fix $R > R_0$, so $P(N=0) = 0$. Suppose $P(N=k) = 1$ for $2 \leq k < \infty$. Let $M_{B(n)}$ be the number of infinite open clusters in $\mathcal{FI}^{u, R^3}$ intersecting $B(n)$. Noting that
	$$
	P \big( M_{B(n)} \geq 2 \big) \xrightarrow{n \rightarrow \infty} P(N \geq 2) =1,
	$$
	there has to be a $n$ such that 
	$$
	P \big( M_{B(n)} \geq 2 \big)>0. 
	$$
	Recall Definition \ref{informal}. Let $F_{1,0}$ be the subgraph in $\zd$ generated by paths starting from $B(n-1)$, $F_{1,1}$ be the subgraph in $\zd$ generated by paths starting from $\partial^{in}B(n)$, and $F_1=F_{1,0}\cup F_{1,1}$. Moreover, let $F_0$ be the subgraph in $\zd$ generated by paths starting from $B^c(n)$. 
	
	Note that $F_{1,0}$ and $F_{1,1}$ may only have countable many configurations, there has to be a pair of (finite) configurations $\CF_{1,0}$ and $\CF_{1,1}$, and a $j\ge 2$ such that 
	$$
	P \big( M_{B(n)}=j, \ F_{1,0}=\CF_{1,0}, \ F_{1,1}=\CF_{1,1} \big)>0,
	$$
	which implies that 
	$$
	P \big(F_0\cup \CF_{1,0}\cup \CF_{1,1} \text{ has $k$ infinite components, among which $j$ components intersect $B(n)$} \big)>0. 
	$$
	We denote the last event by $A_0$ and note that $A_0$ is measurable with respect to $F_0$ and thus independent to $F_{1,0}$ and $F_{1,1}$.
	
	Now let $\hat \CF_{1,1}=\CF_{1,0}\cup \CF_{1,1}\setminus B(n-1)$, and let 
	$$
	\hat \CF_{1,0}=\left\{x\pm \text{e}_j, \ x\in B(n-1), \ j=1,2,\cdots,d \right\}
	$$ 
	be the collection of all edges starting from $B(n-1)$ (or all the edges within $B(n)$). One can immediately see that 
	$$
	P\big(A_0, F_{1,0}=\hat \CF_{1,0}, F_{1,1}=\hat \CF_{1,1}\big)=P\big(A_0\big) P\big(F_{1,0}=\hat \CF_{1,0}, F_{1,1}=\hat \CF_{1,1}\big)>0. 
	$$
	However, given the event above, note that 
	$$
	F_0\cup F_1=F_0\cup \CF_{1,0}\cup \CF_{1,1}\cup \hat \CF_{1,0}. 
	$$
	Since $\hat{\CF}_{1,0}$ contains all the edges within $B(n)$, all the $j$ components in $F_0\cup \CF_{1,0}\cup \CF_{1,1}$ intersecting $B(n)$ merge to one, and the FRI with positive probability only has $k-j+1$ infinite components. This contradicts with $P(N=k)=1$.

	Now suppose $P \big( N = \infty \big) =1$. We say a point $x \in \zd$ is a trifurcation if:
	\begin{enumerate}
		\item $x$ is in an infinite open cluster of $\mathcal{FI}^{u,R^3}$;
		\item there exist exactly three open edges incident to $x$;
		\item removing the three open edges incident to $x$ will split this infinite open cluster of $x$ into exactly three disjoint infinite open clusters. 
	\end{enumerate}
	Define the event $A_x := \{ x \text{ is a trifurcation} \}$. By translation invariance, $P(A_x)$ is constant for all $x \in \zd$. Therefore,
	$$
	\frac{1}{|B(n)|} E \Bigg[ \sum_{x \in B(n)} \ind_{A_x} \Bigg] = P (A_0).
	$$
	Recall that $M_{B(n)}$ is the number of infinite open clusters in $\mathcal{FI}^{u, R^3}$ intersecting $B(n)$. Note that
	$$
	P \big( M_{B(n)} \geq 3 \big) \xrightarrow{n \rightarrow \infty} P(N \geq 3) =1.
	$$
	Define the event
	$$
	E_n := \Big\{ \text{No killed random walks starting in } \zd \setminus B(2n) \text{ intersects } B(n) \Big\}.
	$$
	By Lemma \ref{no outside}, the probability of event $E_n^{c}$ decays stretch exponentially. We can choose $n$ large enough such that 
	$$
	P \big( M_{B(n)} \geq 3, E_n \big) > 1/2.
	$$
	Similarly, let $F_1$ and $F_2$ be the random subgraphs in $\zd$ generated by the trace of all killed random walks starting in $B(n)$ and $B(2n) \setminus B(n)$, respectively. Note that $F_1$ and $F_2$ are independent. Since there are only countably many choices for $F_1$ and $F_2$, there exist two finite subgraphs $\mathcal{F}_1$ and $\mathcal{F}_2$ in $\zd$ such that
	$$
	P \big( M_{B(n)} \geq 3, E_n , F_1 = \mathcal{F}_1, F_2 = \mathcal{F}_2 \big) > 0.
	$$
	If $\omega \in \{ M_{B(n)} \geq 3, E_n , F_1 = \mathcal{F}_1, F_2 = \mathcal{F}_2 \}$, then there exist $x(\omega), y(\omega), z(\omega) \in \partial^{in} B(n)$ lying in three distinct infinite open clusters in $\zd \setminus B(n)$. There are three paths connecting the origin and $x,y,z$, respectively, in the following way:
	\begin{enumerate}
		\item $0$ is the unique common vertex in any two paths;
		\item each path touches exactly one vertex in $\partial^{in} B(n)$.
	\end{enumerate}
	Let $D_{x,y,z,n}$ be the event that:
	\begin{enumerate}
		\item there are exactly three killed random walks starting at the origin;
		\item these three killed random walk paths end at $x,y,z$, respectively, and they satisfy the conditions above;
		\item  no killed random walks start at any vertices in $B(n) \setminus \{0\}$. 
	\end{enumerate}
	It is easy to see that $P(D_{x,y,z,n}) >0$ for all $n>0$ and all distinct $x,y,z \in \partial^{in} B(n)$. Since $\mathcal{F}_1$ and $\mathcal{F}_2$ are fixed and finite,
	$$
	P\Big( F_2 = \mathcal{F}_1 \cup \mathcal{F}_2 \setminus B(n) \Big) >0.
	$$
	For $\omega \in \{ M_{B(n)} \geq 3, E_n , F_1 = \mathcal{F}_1, F_2 = \mathcal{F}_2 \}$, we can resample all $N_x$ for $x \in B(2n)$, and then we resample all killed random walk paths starting in $B(2n)$ accordingly. Note that the resulting graph is still distributed as FRI $\mathcal{FI}^{u,R^3}$. If the events $D_{x,y,z,n}$ and $\{F_2 = \mathcal{F}_1 \cup \mathcal{F}_2 \setminus B(n)\}$ occur after the resample, then $0$ is a trifucation. Therefore,
	$$
	P \big( A_0 \big) \geq P \Big( D_{x,y,z,n} \Big) P\Big( F_2 = \mathcal{F}_1 \cup \mathcal{F}_2 \setminus B(n) \Big) P \Big( M_{B(n)} \geq 3, E_n , F_1 = \mathcal{F}_1, F_2 = \mathcal{F}_2 \Big) >0.  
	$$
	Now we can adapt the proof of Theorem $12.2$ in \cite{haggstrom2006uniqueness} (or the proof of Burton and Keane \cite{Burton1989Density} if one considers a site percolation on $\mathcal{FI}^{u,R^3}$). For each trifurcation $t \in B(n)$, there is a one-to-one corresponding point $y_t \in \partial^{in} B(n)$. However, the number of trifurcation points grow in $B(n)$ as $n^d$, but $\partial^{in} B(n)$ grows as $n^{d-1}$. We have a contradiction.   
\end{proof}

\section{Subcritical Phase}
\label{sec subcritical}

In this section we present the proof of Theorem \ref{main2}.

\begin{proof}[Proof of Theorem \ref{main2}] 
We use the Peierls argument \cite{R2008On}. Fix $u >0$. Let $\mathcal{C}$ be the connected component that contains the origin in the FRI, $\mathcal{FI}^{u,T}$. It suffices to show that there is a constant $T_0 (u) >0$ such that for all $0< T < T_0$,
$$
P \big( |\mathcal{C}| = \infty \big) = 0.
$$
We say a path is self-avoiding if it does not visit the same edge twice. Note that the number of self-avoiding paths in $\zd$ which have length $n$ and start at the origin is bounded above by $(2d)^n$. Let $N(n)$ be the number of such paths which are open. If the origin belongs to an infinite open cluster, then there are open self-avoiding paths starting at the origin of all lengths. So for all $n >0$,
$$
P \big( |\mathcal{C}| = \infty \big) \leq P \big( N(n) \geq 1 \big) \leq E \big[ N(n) \big].
$$  
Let $\gamma$ be a self-avoiding path that has length $n$ and starts at the origin. We want to estimate the probability that $\gamma$ is open. Let $N_{\gamma}$ be the number of killed random walks that traverse $\gamma$. Recall that $N_{\gamma}$ is a Poisson random variable with parameter $u \cdot \text{cap}^{(T)} (\gamma)$. Since the path $\gamma$ has length $n$, it has $n+1$ vertices. Note that the killed equilibrium measure is always less than or equal to $2d$, so
$$
\text{cap}^{(T)} (\gamma) \leq 2d(n+1),
$$
for all $T>0$. By exponential Markov inequality,
\begin{equation}
\begin{aligned}
& P \Big( N_{\gamma} > eu(2d)(n+1)+ (n+1) \log (3d) \Big) \\
& \leq \frac{E \big[e^{N_{\gamma}}\big]}{\exp \big( eu(2d)(n+1)+ (n+1) \log (3d) \big)} \\
& = \frac{\exp \big(u (e-1) \cdot \text{cap}^{(T)} (\gamma) \big)}{\exp \big( eu(2d)(n+1)+ (n+1) \log (3d) \big)} \\
& \leq \exp \big( eu(2d)(n+1) - eu(2d)(n+1) - (n+1) \log (3d) \big) \\
& = (3d)^{-n-1}.
\end{aligned}
\end{equation}
If the path $\gamma$ is open in $\mathcal{FI}^{u,T}$, then the $N_{\gamma}$ killed random walks that traverse $\gamma$ must travel more than $n$ steps in total after they first enter $\gamma$. Assume $0< T < 1$. Note that the survival rate for killed random walks at each step is $T/(T+1)$, which is smaller than $T$. Let $Y_1, Y_2, \cdots$ be i.i.d. geometric random variables with parameter $1-T$. Let
$$
L:= \lceil eu(2d)(n+1)+ (n+1) \log (3d) \rceil. 
$$
Then,
$$
P \big( \gamma \text{ is open} \big| N_{\gamma} \leq L \big) \leq P \Bigg( \sum_{i=1}^{L} Y_i \geq L + n \Bigg).
$$ 
By Chernoff bound,
$$
P \Bigg( \sum_{i=1}^{L} Y_i \geq L + n \Bigg) \leq e^{-t(L+n)} \Bigg( \frac{(1-T)e^t}{1 - T e^t}\Bigg)^{L} = e^{-tn} \Bigg( \frac{1-T}{1 - T e^t} \Bigg)^{L},
$$
for all $t>0$ such that $Te^t<1$. Take $t_0 = \log (6d)$. We choose $0< T_0(u) <1 $ such that
$$
T_0 e^{t_0} = 6d T_0 < 1,
$$
and
$$
\Bigg( \frac{1-T_0}{1 - T_0 e^{t_0} }\Bigg)^{ \lceil eu(2d)+  \log (3d) \rceil } \leq 2.
$$
Then for all $0 < T < T_0$,
$$
P \big( \gamma \text{ is open} \big| N_{\gamma} \leq L \big) \leq e^{-t_0 n} \Bigg( \frac{1-T}{1 - T e^{t_0}} \Bigg)^{L} \leq (6d)^{-n} 2^{n+1} = 2(3d)^{-n}. 
$$
So,
$$
P \big( \gamma \text{ is open} \big)  \leq P \big( \gamma \text{ is open} \big| N_{\gamma} \leq L \big) + P \big( N_{\gamma} > L \big)  \leq 2(3d)^{-n} + (3d)^{-n-1}.
$$
Since $\gamma$ is arbitrary,
$$
P \big( |\mathcal{C}| = \infty \big) \leq E \big[ N(n) \big] \leq (2d)^{n} \Big( 2(3d)^{-n} + (3d)^{-n-1} \Big) \xrightarrow{n \rightarrow \infty} 0.
$$
The proof is complete.
\end{proof}

\section*{Acknowledgments}
We would like to thank Dr. Bal\'{a}zs R\'{a}th for reminding us relevant studies that we had first overlooked. Part of this paper was written while the first two authors were visitors of Peking University. We would like to thank an anonymous referee for very detailed and helpful comments.

\bibliographystyle{plain}
\bibliography{cite}

\end{document}